\documentclass[a4paper,12pt]{amsart}

\usepackage[]{graphicx,color}
\usepackage[colorlinks=false]{hyperref}
\usepackage[paperwidth=210mm,paperheight=297mm,inner=2.8cm,outer=2.8cm,top=2.5cm,bottom=2.6cm]{geometry}
\usepackage{multicol}
\usepackage{multirow}

\theoremstyle{plain}
\newtheorem{Theorem}{Thm}[section]
\newtheorem{Thm}[Theorem]{Theorem}
\newtheorem{Lem}[Theorem]{Lemma}
\newtheorem{Cor}[Theorem]{Corollary}
\newtheorem{Prop}[Theorem]{Proposition}

\newtheorem*{Thm*}{Theorem}
\newtheorem*{Lem*}{Lemma}
\newtheorem*{Rem*}{Remark}
\newtheorem*{Conj}{Conjecture}
\theoremstyle{definition}
\newtheorem{Def}[Theorem]{Definition}
\newtheorem{Exm}[Theorem]{Example}
\newtheorem{Rem}[Theorem]{Remark}

\newcommand\C{\mathbb{C}}

\renewcommand\P{\mathbb{P}}

\newcommand\R{\mathbb{R}}
\newcommand\Z{\mathbb{Z}}
\newcommand\V{\mathcal{V}}
\newcommand\sH{\mathcal{H}}

\DeclareMathOperator\lspan{span}
\DeclareMathOperator\conv{conv}

\DeclareMathOperator\codim{codim}
\DeclareMathOperator\mim{im}

\DeclareMathOperator\rank{rank}

\newcommand\ol{\overline}

\newcommand\into{\hookrightarrow}

\renewcommand\div{ {\rm div}}

\newcommand\acm{arithmetically Cohen-Macaulay}

\usepackage{paralist}


\begin{document}
\title[Low-rank sum-of-squares representations]{Low-rank sum-of-squares representations on varieties of minimal degree}
\author{Grigoriy Blekherman}
\address{Georgia Institute of Technology, Atlanta GA }
\email{greg@math.gatech.edu }
\author{Daniel Plaumann}
\address{Technische Universit\"at Dortmund, Dortmund, Germany }
\email{Daniel.Plaumann@math.tu-dortmund.de}
\author{Rainer Sinn}
\address{Georgia Institute of Technology, Atlanta, GA, USA }
\email{rsinn3@math.gatech.edu }
\author{Cynthia Vinzant}
\address{North Carolina State University, Raleigh, NC, USA }
\email{clvinzan@ncsu.edu}

\begin{abstract}
A celebrated result by Hilbert says that every real nonnegative ternary quartic is a sum of three squares. We show more generally that every nonnegative quadratic form on a real projective variety $X$ of minimal degree is a sum of ${\dim(X)+1}$ squares of linear forms.
This strengthens one direction of a recent result due to Blekherman, Smith, and Velasco. Our upper bound is the best possible, and it implies the existence of low-rank factorizations of positive semidefinite bivariate matrix polynomials and representations of biforms as sums of few squares.
We determine the number of equivalence classes of sum-of-squares representations of general quadratic forms on surfaces of minimal degree, generalizing the count for ternary quartics by Powers, Reznick, Scheiderer, and Sottile.
\end{abstract}
\maketitle

\section*{Introduction}
The relationship between nonnegative polynomials and sums of squares
is a fundamental question in real algebraic geometry. It was first
studied by Hilbert in an influential paper from 1888. He showed that every
nonnegative homogeneous polynomial in $n$ variables of degree $2d$ is
a sum of squares in the following cases only:
bivariate forms ($n=2$), quadratic forms ($2d=2$), and ternary
quartics ($n=3$, $2d=4$).

In the case of ternary quartics, Hilbert showed that every nonnegative
polynomial is a sum of at most three squares. This bound is sharp: a
general nonnegative ternary quartic is not a sum of two squares. This
result has attracted attention over the years; see \cite{Rajwade,
 Rudin, ScheidererMR2580677, SwanMR1803372}. The most recent
elementary proof is due to Pfister and Scheiderer
\cite{PfisterScheiderer}. A different proof was given by Powers,
Reznick, Scheiderer, and Sottile
\cite{PowersReznickScheidererSottileMR2103198}. Additionally, they
showed that a general ternary quartic is a sum of three squares in
precisely $63$ essentially different ways over $\C$, and a general
nonnegative ternary quartic is a sum of three squares in $8$ ways over
$\R$. The number $63$ was also obtained by Plaumann, Sturmfels, and
Vinzant \cite{PSV}. In the two other cases of Hilbert's theorem, 
bounds on the number of squares are well-known: every nonnegative bivariate
form is a sum of at most two squares and a quadratic form in $n$
variables is a sum of at most $n$ squares. Both bounds are sharp
generically.

Recently, Blekherman, Smith, and Velasco generalized Hilbert's theorem
to polynomials nonnegative on an irreducible variety $X$ with dense
real points \cite{BSV}. By considering the $d$-th Veronese embedding
$\nu_d(X)$ of $X$, we may reduce the case of polynomials of degree
$2d$ nonnegative on $X$ to quadratic forms on $\nu_d(X)$. Therefore,
it suffices to classify all real varieties $X$ on which all
nonnegative quadratic forms are sums of squares. It was shown in
\cite{BSV} that these are exactly the varieties of minimal degree,
which were classified by Del Pezzo and Bertini; see
\cite{EisenbudHarris} for a modern exposition.

We provide a strengthening and a new proof
of one direction of the main theorem of Blekherman, Smith, and
Velasco. Our first main result generalizes the previous work on the
three cases of equality in Hilbert's theorem to varieties of minimal
degree. 

\begin{Thm*}[Theorem~\ref{Thm:varmin}]
Let $X\subset \P^n$ be a nondegenerate irreducible variety of minimal
degree with dense real points. Then every quadratic form nonnegative
on $X$ is a sum of $\dim(X)+1$ squares in the homogeneous coordinate ring $\R[X]$. 
\end{Thm*}
The above bound is sharp in the sense that a general quadratic form
nonnegative on $X$ is not a sum of fewer than $\dim(X)+1$ squares.
This theorem gives a unified proof of this bound, which was proved independently with different techniques for the different families of varieties of minimal degree. Our proof follows a line of reasoning similar to that of Hilbert's original proof for ternary quartics.

In the case that the variety $X$ is a rational normal scroll, this
theorem has an elegant interpretation from the point of view of
non-commutative real algebraic geometry. It gives a tight
Positivstellensatz for homogeneous bivariate matrix polynomials.

\begin{Thm*}[Corollary~\ref{cor:noncomm}]
Let $A$ be a symmetric $n\times n$ matrix whose entries are homogeneous polynomials in two variables $s$ and $t$. Suppose that $A(s,t)$ is a positive semidefinite matrix for every $(s,t)\in\R^2$. Then there is a matrix $B$ of size $n\times(n+1)$ 
with entries in $\R[s,t]$ such that $A = B B^T$.
\end{Thm*}

A bound of $2n$ instead of $n+1$ was shown by Choi, Lam, and Reznick \cite{ChoiLamReznickMR566480}, among others \cite{JakubovicMR0271129, RosenblumRovnyakMR0273437}.
The improvement to $n+1$, which is tight generically, was observed by Leep in \cite{Leep} using techniques from the theory of quadratic forms. Our approach shows furthermore that there are only finitely many inequivalent representations as sums of $n+1$ squares of generic nonnegative quadratic forms on a rational normal scroll.

We also extend the result of Powers, Reznick, Scheiderer, and Sottile in \cite{PowersReznickScheidererSottileMR2103198} on the number of inequivalent representations as a sum of three squares to all surfaces of minimal degree. 
By the classification of varieties of minimal degree, a surface of minimal degree is either a quadratic hypersurface in $\P^3$, the Veronese surface in $\P^5$, corresponding to ternary quartics, or a rational normal scroll. 

\begin{Thm*}[{Theorem~\ref{thm:allSurfaces}}]
Let $X \subset \P^n$ be a nondegenerate irreducible surface of minimal degree with dense real points. Then a generic quadratic form nonnegative on $X$ has exactly $2^{n-2}$ inequivalent representations as a sum of three squares.
\end{Thm*}

The case of the Veronese surface was already solved in \cite{PowersReznickScheidererSottileMR2103198}, so we concentrate on the rational normal scrolls.
We also count the number of representations of a general quadratic form as a sum and difference of $\dim(X)+1$ squares over $\R$ and the number of representations as a sum of $\dim(X)+1$ squares over $\C$.

\begin{Thm*}[{Theorems~\ref{Thm:complexcount} and \ref{Thm:realsosrep}}]
Let $f$ be a generic quadratic form on a two-dimensional smooth real rational normal scroll $X\subset \P^n$. Then $f$ has exactly $2^{2(n-2)}$ inequivalent representations as a sum of three squares over $\C$. If $n$ is even, then all real representations of $f$ are sums of three squares and there are $2^{n-2}$ inequivalent such representations. If $n$ is odd, then there are $2^{n-1}$ inequivalent real representations, with $2^{n-2}$ as sums of three squares and $2^{n-2}$ as sums and differences of three squares.
\end{Thm*}
In the proof, we relate representations as sums of three squares to two-torsion points on the Jacobian of the smooth curve associated to a generic quadratic form. This extends an observation due to Coble \cite{CobleMR733252}, also employed in \cite{PowersReznickScheidererSottileMR2103198}. For ternary quartics, the smoothness of the curve makes the form sufficiently generic for the count to work. The case of rational normal scrolls is more delicate and we need to make further genericity assumptions, as illustrated by Example~\ref{Exm:generic}.

For higher dimensions, we conjecture the following.
\begin{Conj}
Let $X\subset \P^n$ be a nondegenerate irreducible variety of minimal degree with dense real points. Then a generic quadratic form nonnegative on $X$ has exactly $2^{\codim(X)}$ inequivalent representations as a sum of $\dim(X)+1$ squares.
\end{Conj}

The conjecture holds for $\dim(X)=1$, by \cite[Example 2.13]{ChoiLamReznickMR1327293}, 
and $\dim(X)=2$, by Theorem~\ref{thm:allSurfaces}. 
For threefolds, we have some computational evidence.

We find it remarkable that the number of sum-of-squares representations over $\C$
is not as regular as the number of sum-of-squares representations over $\R$, see Example~\ref{exm:repsforsurf}. 
Algebraic intuition would suggest the exact opposite, but we have no general conjecture 
that would incorporate the number 63 from the case of ternary quartics.

\bigskip \textit{Acknowledgments:} We would like to thank Claus
Scheiderer for fruitful discussions and the referee for helpful comments
on the presentation. Grigoriy Blekherman and Rainer
Sinn were supported by NSF grant DMS-0757212. Daniel Plaumann
gratefully acknowledges financial support from the Zukunftskolleg of
the University of Konstanz and DFG grant PL 549/3-1. Cynthia Vinzant was supported by NSF grant DMS-1204447 and the FRPD program at North Carolina State University. 

\section{The Minimal Length of Sum-of-Squares Representations}
Let $X\subset \P^n$ be an irreducible real projective variety of
dimension $m$. Assume that $X(\R)$ is Zariski-dense in $X$ and that
$X$ is nondegenerate (i.e.~not contained in a proper subspace) and of
minimal degree (i.e.~$\deg(X)=\codim(X)+1$).

\begin{Thm}\label{Thm:varmin}
Every quadratic $f\in\R[X]_2$ such that $f(x)\geq 0$ for every $x\in X(\R)$ is a sum of $(\dim(X) + 1)$ squares of linear forms in $\R[X]_1$.
\end{Thm}

In the proof below, we generalize Hilbert's proof of the fact that every nonnegative ternary quartic is a sum of three squares, 
see \cite{HilbertMR1510517}. For a rigorous and modern presentation of Hilbert's proof, see Swan \cite{SwanMR1803372}.
We consider the map 
\begin{equation}
\phi \ \ \colon \left\{
  \begin{array}{l}
\R[X]_1 \times \dots \times \R[X]_1 \to \R[X]_2\\
 (\ell_0,\hdots,\ell_{m})  \ \mapsto \  \sum_{i=0}^{m} \ell_i^2
  \end{array}
\right. \label{eq:phi}
\end{equation}
The theorem is equivalent to the statement that the image of this map is equal to the convex cone of nonnegative quadratic forms in $\R[X]_2$. 
\begin{Lem}\label{lem:proper} Let $\phi$ be the map defined in \eqref{eq:phi}.
\begin{itemize}
\item[(a)] The map $\phi$ is proper and closed.
\item[(b)] The differential of $\phi$ is surjective at every point
$(\ell_0,\hdots,\ell_{m})$ that gives a base-point-free linear system on
$X$, i.e.~for which $X\cap \V(\ell_0,\hdots,\ell_{m})$ is empty.
\end{itemize}
\end{Lem}

\begin{proof}
(a) The map $\phi$ is homogeneous and $\phi(\ell_0,\hdots,\ell_{m})\neq 0$ whenever $(\ell_0,\hdots,\ell_{m})\neq 0$. So we can view it as a continuous map from $\P(\R[X]_1\times\dots\times\R[X]_1)$ to $\P(\R[X]_2)$, where we take the Euclidean topology on both projective spaces. As a continuous map between compact Hausdorff spaces, it is both proper and closed.\\
(b) The differential at $(\ell_0,\ldots,\ell_{m})$ is the map
\begin{equation*}
{\rm d}\phi \colon \left\{
  \begin{array}{l}
\R[X]_1  \times \dots \times \R[X]_1 \to \R[X]_2 \\
(h_0,\hdots,h_{m}) \ \mapsto \ 2\sum_{i=0}^{m} h_i\ell_i.
  \end{array}\right.
\end{equation*}
We can count the dimension of the image by counting the syzygies among the linear forms $\ell_0,\ldots,\ell_{m}$. The assumption $X\cap \V(\ell_0,\hdots,\ell_{m}) = \emptyset$ implies that $\ell_0,\hdots,\ell_{m}$ is a homogeneous system of parameters in $\R[X]$. Since $X$ is {\acm} by \cite[Theorem 4.2]{EisGotoMR741934}, this homogeneous system of parameters is also a regular sequence. So the only syzygies among the $\ell_i$ are the obvious ones
$\ell_i \ell_j = \ell_j \ell_i$ for $i\neq j$. 
 Therefore, the rank of the differential at $(\ell_0, \hdots, \ell_m)$ is
\begin{equation}\label{eq:dimcount}
(m+1)\dim(\R[X]_1) - \binom{m+1}{2} \ = \ (m+1)(n+1) - \binom{m+1}{2} \ = \ \dim(\R[X]_2).
\end{equation}
The last equality holds because $X$ is a variety of minimal degree and therefore, its quadratic deficiency $\epsilon(X)$ is $0$; see Blekherman-Smith-Velasco \cite[Section~3]{BSV}.
\end{proof}

\noindent We can now finish the proof with the same topological argument used by Hilbert.
\begin{proof}[Proof of Theorem~\ref{Thm:varmin}]
Let $\mathcal{P}\subset \R[X]_2$ be the set of all strictly positive quadrics $q$ such that
$X\cap \V(q)$ is smooth. It is open and connected because the set of all strictly positive quadrics
with a complex singularity has codimension at least two in $\R[X]_2$, as any such quadric must also be singular
at the complex conjugate of the singularity.  
By Lemma~\ref{lem:proper}, the set $\mathcal{P}\cap \mim(\phi)$ is a closed subset of $\mathcal{P}$. On the other hand, it is also open in $\mathcal{P}$ because every $q\in \mathcal{P}\cap \mim(\phi)$ is the sum of squares of a regular sequence. Indeed, if $\ell_0,\ldots,\ell_m$ have a common zero on $X$, then $q=\ell_0^2+\hdots+\ell_m^2$ would be singular at this point. So $q$ is an interior point of $\mathcal{P}\cap \mim(\phi)$ by the implicit function theorem and Lemma~\ref{lem:proper}(b). Since $\mathcal{P}$ is connected, we conclude $\mathcal{P}\subset \mim(\phi)$. Since $\mathcal{P}$ is dense in the cone of nonnegative polynomials and $\mim(\phi)$ is closed, we conclude that every nonnegative quadratic is a sum of $m+1$ squares of linear forms.
\end{proof}

\begin{Def}\label{def:Equiv}
  We say that two representations
  $q=\ell_0^2+\hdots+\ell_m^2=\ell_0^{\prime 2}+\hdots+\ell_m^{\prime 2}$ are
  \emph{equivalent} if there exists an orthogonal $(m+1)\times(m+1)$ matrix $O$ such that
\[
\begin{pmatrix}
\ell_0 & \ell_1 & \hdots & \ell_m
\end{pmatrix}^T = 
O \begin{pmatrix}
\ell_0' & \ell_1' & \hdots & \ell_m'
\end{pmatrix}^T.
\]
The equivalence of sum-of-squares representations,
interpreted as quadratic forms, can also be understood in terms of
their representing matrices. Explicitly, a \emph{Gram matrix} of $q\in \R[X]_2$ 
is a $(n+1)\times (n+1)$ symmetric matrix $A$ for which 
\[
q \ =  \ (x_0,\hdots,x_n) A (x_0,\hdots,x_n)^T \ \ \text{ in } \ \  \R[X]_2.
\]
\end{Def}
If $A$ is positive semidefinite and has rank $m+1$, then $A$ decomposes as $A=BB^T$ with $B$ of
size $(n+1)\times (m+1)$. This gives rise to a representation of $q$ as a sum of $m+1$ squares. 
Two representations are equivalent if and only if they come from two decompositions of the same Gram
matrix. One can check that an indefinite Gram matrix corresponds to an equivalence class of
representations as a sum and difference of squares, and Gram matrix with complex entries corresponds to an equivalence class of 
representations as a sums of squares over $\C$. For more details, see \cite{GramSpec}.

\begin{Cor}\label{cor:finitereps}
There are finitely many equivalence classes of representations of a general quadratic $q\in \mathcal{P}$ as a sum of $m+1$ squares.
That is, a generic $q\in \mathcal{P}$ has finitely many Gram matrices of rank $(m+1)$. 
\end{Cor}

\begin{proof}
  The dimension count in equation (\ref{eq:dimcount}) at the end of
  the proof of Lemma~\ref{lem:proper} shows that the fiber
  $\phi^{-1}(q)$ of a general quadratic form $q\in\mathcal{P}$ has
  dimension $\binom{m+1}{2}$, see for example \cite[\S 2, Lemma
  1]{MilnorMR1487640}. Since the orthogonal group of
  $(m+1)\times (m+1)$ matrices has dimension $\binom{m+1}{2}$ and acts
  faithfully on linearly independent linear forms, we have only
  finitely many orbits in such a fiber. The equivalent statement for Gram matrices 
  follows from the discussion above. 
\end{proof}

\begin{Cor}[to Theorem~\ref{Thm:varmin}]\label{cor:noncomm}
Let $A$ be a symmetric $n\times n$ matrix whose entries are homogeneous polynomials in two variables $s$ and $t$. Suppose that $A(s,t)$ is a positive semidefinite matrix for every $(s,t)\in\R^2$. Then there is a matrix $B$ of size $n\times(n+1)$ with polynomial entries such that $A = B B^T$.
\end{Cor}

\begin{proof}
This follows from Theorem~\ref{Thm:varmin} with the following observation:
If the $i$th diagonal entry has degree $2 d_i$ for $i=1,\hdots,n$, then the $(i,j)$th entry must have degree $d_i+d_j$, because $A$ is everywhere positive semidefinite. This can be proved by looking at symmetric $2\times 2$ minors of $A$. Therefore, $A$ defines a quadratic form on a rational normal scroll of dimension $n$. This variety of minimal degree can be realized as a toric variety whose corresponding polytope is a truncated prism over the $(n-1)$-dimensional standard simplex with heights $d_i$ at the vertices.
\end{proof}

\section{The number of representations for surfaces}\label{sec:count}
The smooth surfaces of minimal degree are the quadratic hypersurfaces in $\P^3$, the Veronese surface in $\P^5$ corresponding to ternary quartics, and the two-dimensional rational normal scrolls, which are toric embeddings of the Hirzebruch surfaces.
These toric embeddings are given by special lattice polytopes $P\subset\R^2$. The corresponding polynomials are biforms of bidegree $(2,2d)$, which are polynomials whose Newton polytope is contained in $2P$. 
For our count of the number of sum-of-squares representations of
biforms of bidegree $(2,2d)$, we relate such representations to
two-torsion points on the Jacobian of the hyperelliptic curve defined
by a biform in $\P^1\times \P^1$.

Following the approach by toric geometry in \cite{CoxLittleSchenckMR2810322}, we identify a smooth 
rational normal scroll by two positive integers $d \geq e$, which define the polytope
\[
P_{d,e}  \ = \  \conv\{(0,0),(0,1),(d,0),(e,1)\}.
\]
The associated projective toric variety is the Zariski closure of the image of the map
\begin{equation}\label{eq:toric}
(\C^\ast)^2 \to \P^{d+e+1}, \ \  (s,x)\mapsto (1:s:s^2:\ldots:s^d:x:xs:xs^2:\ldots:xs^e).
\end{equation}
We can count the number of representations of biforms with Newton polytope $2P_{d,e}$ in terms of these defining positive integers $d$ and $e$.  
We begin by counting the representations over the complex numbers.

\begin{Thm}\label{Thm:complexcount}
Let $d\geq e$ be positive integers and let $f$ be a generic polynomial with Newton polytope $2 P_{d,e}$. 
Then over $\C$, $f$ has exactly $2^{2 g}$ inequivalent representations as a sum of three squares of forms with Newton polytope $P_{d,e}$, where $g = d+e-1$.
\end{Thm}

Theorem~\ref{Thm:complexcount} is proved on page~\pageref{proof:complexcount}. Note that the polynomial $f$ is not homogeneous. We bi-homogenize $f$ in $s$ and $x$ with homogenizing variables $t$ and $y$ to a biform of degree $2d$ in $s$ and $t$ and degree two in $x$ and $y$. In matrix form, this means
\begin{equation}\label{eq:biform}
f \  =  \
\begin{pmatrix}
x & y
\end{pmatrix}
\begin{pmatrix}
a(s,t) & b(s,t) \\
b(s,t) & c(s,t)
\end{pmatrix}
\begin{pmatrix}
x \\ y
\end{pmatrix},
\end{equation}
where $a$, $b$, and $c$ are bivariate forms of degree $2d$, $a$ is divisible by $t^{2(d-e)}$, and $b$ is divisible by $t^{d-e}$.
This biform defines a curve in $\P^1 \times \P^1$. If $d\neq e$, then a generic biform with Newton polytope $2P_{d,e}$ defines a singular curve in $\P^1\times \P^1$. We can embed the smooth model of this curve in the toric variety associated 
with the polytope $P_{d,e}$.  
\begin{Lem}\label{lem:curve}
Let $f$ be a generic biform with Newton polytope $2P_{d,e}$. The smooth model of the curve $\V(f)\subset\P^1\times \P^1$ has genus $g = d+e-1$ and can be embedded as a curve of degree $2(d+e)$ in $\P^{d+e+1}$ as the intersection of the rational normal scroll defined by $P_{d,e}$ and a quadric given by $f$. This embedding of the curve is {\acm} and projectively normal.
\end{Lem}

\begin{proof}
Set $P = P_{d,e}$ and let $X_P\into \P^{d+e+1}$ be the projective toric variety defined by $P$ as in \eqref{eq:toric}.
The Hilbert polynomial of the surface $X_P$ is equal to the Ehrhart polynomial of the polytope $P$, see \cite[Proposition 9.4.3 and Corollary 2.2.19]{CoxLittleSchenckMR2810322}, and the Ehrhart polynomial of $P$ is
\[
p_X(\mathfrak{t})  \ = \  \frac12 (d+e) \mathfrak{t}^2 + \frac12 (d+e+2) \mathfrak{t} + 1.
\]
The coefficients of $f$ define a quadric in $\P^{d+e+1}$. Let $C$ be the intersection of $X_P$ with this quadric. Since $f$ is generic, $C$ is smooth and nondegenerate by Bertini's theorem \cite[Th\'eor\`eme 6.2]{JouMR725671}. Since $\V(f)\subset \P^1\times\P^1$ and $C$ are birational and $C$ is smooth, $C$ is indeed an embedding of the smooth model of $\V(f)$.

We can compute the genus and degree of $C$ by computing the Hilbert polynomial $p_C$ of the curve, which is
\[
p_C(\mathfrak{t}) \ =  \ p_X(\mathfrak{t}) - p_X(\mathfrak{t}-2)  \ = \ 2(d+e) \mathfrak{t} + (2-d-e).
\]
So the genus of $C$ is $g = d+e-1$ and the degree is $2(d+e)$. 

The curve is \acm, because the toric surface $X_P$ is by \cite[Exercise 9.2.8]{CoxLittleSchenckMR2810322}. Every curve that is {\acm} is projectively normal, see \cite[Exercise 18.16]{EisenbudMR1322960}. 
\end{proof} 

From now on, we fix the integers $d$ and $e$ and simply write $P$ for
$P_{d,e}$. We identify a biform $f$ with Newton polytope $2P$ with a
quadratic form in $\C[X_P]_2$ and the smooth model of the curve
$\V(f)\subset \P^1 \times \P^1$ with the intersection $C$ of $X_P$ and
the quadratic form in $\C[X_P]$ corresponding to $f$. We also identify
biforms with Newton polytope $P$ with linear forms in $\C[X_P]_1$. So
a representation $f = \ell_1^2 + \ell_2^2 + \ell_3^2$ of a biform is a
representation of the quadratic form $f$ as a sum of three squares of
linear forms in $\C[X_P]$. Over the complex numbers, such a
representation is equivalent to $f = pq+r^2$, where
$p = (\ell_1 + i \ell_2)$, $q = (\ell_1 - i \ell_2)$, and
$r = \ell_3$. We will consider this type of representation, sometimes called a quadratic representation of $f$, from now on.

Given a linear form $p\in\C[X_P]_1$, the intersection of $C$ with the
hyperplane given by $p$ defines a divisor on $C$ which we denote $\div_C(p)$.

\begin{Lem}\label{lem:Divisor}
  If $f=pq+r^2$, the divisor $\div_C(p)$ is
  even, i.e.~there exists a divisor $D$ on $C$ such that $\div_C(p) = 2 D$, and the linear system $|\frac 12\div_C(p)| = |D|$ is base-point-free.
\end{Lem}

\begin{proof}
  The identity $f=pq+r^2$ translates into an identity of
  divisors
\[
  \div_C(p)+\div_C(q) \ = \ 2\div_C(r).
\]
  on $C$. Since $C$ is smooth, the supports of the divisors $\div_C(p)$ and
  $\div_C(q)$ are disjoint, as any common zero of $p$ and $q$
  on $C$ would also be a zero of $r$ and therefore a zero of $f$ with multiplicity $\geq$ 2. Hence $\div_C(p)$ is
  even. Also, $\frac 12\div_C(p)$ and $\frac 12\div_C(q)$ are
  linearly equivalent, since 
\[
  \frac 12\div_C(p)-\frac 12\div_C(q) \ = \ \div_C(r/q).
\]
  This shows that the linear system $|\frac 12\div_C(p)|$ is base-point-free.
\end{proof}

The proof of Theorem~\ref{Thm:complexcount} relies on a converse of Lemma~\ref{lem:Divisor}.
  To build up to this, we give an identification between representations $pq + r^2 = \gamma \cdot f$
  and two-torsion points of the Jacobian of the curve $C$, \textit{i.e.} divisor classes $[E]$ with $2E\sim 0$. 
  When $\gamma$ is non-zero, this rescales to give a representation of the form $pq + r^2 = f$. 

 \begin{Lem}\label{lem:Divisor2Biform}
Suppose $f = p_0 q_0 + r_0^2$  and $[E]$ is a two-torsion point of the Jacobian of the curve $C$. 
Then there exist linear forms $p, q,r$ on $\P^{d+e+1}$ with $p\neq q$ such that 
$pq+r^2 \  = \  \gamma \cdot  f $ for some $\gamma \in \C$, the divisor $\div_C(p)$ is even, and
$\frac{1}{2}\div_C(p)$ is linearly equivalent to $E + \frac{1}{2}\div_C(p_0)$.
\end{Lem}

\begin{proof}
By Lemma~\ref{lem:Divisor}, the divisor
$\div_C(p_0)$ is even, and we define $D_0 = \frac12 \div_C(p_0)$. 
Note that the degree of $D_0$ is $\deg(C)/2 = d+e$. As above, the curve $C$ has genus $g = d+e-1$. 
Then the Riemann-Roch theorem shows that 
\[
l(E + D_0) \  \geq  \ \deg(D_0) + 1 - g  \ = \  (d+e)+1-(d+e-1)  \ = \ 2,
\]
where $l(E+D_0)$ is the dimension of the vector space underlying the
linear system $|E+D_0|$ (see~\cite[Section 8.6]{Fulton}).
Therefore, there are two distinct effective divisors $D,D'$ in $|E + D_0|$. The three divisors $2D$, $2D'$, and $D+D'$ are linearly equivalent to $2D_0 = \div_C(p_0)$. Since the curve $C\subset\P^{d+e+1}$ is projectively normal, there are linear forms $p$, $q$, and $r$ on $\P^{d+e+1}$ such that
\[
2D  \ = \  \div_C(p),   \ \ 2D'  \ = \  \div_C(q), \ \ D + D'  \ = \  \div_C(r).
\]
This implies that
\[
\div_C\left(\frac{pq}{r^2}\right)  \ = \  0, 
\]
so the rational function $pq/r^2$ is constant on $C$. After multiplying $r$ by a scalar, we can assume that
\[
pq \  = \  - r^2 \ \ \ \text{ in }  \ \ \ \C[C] = \C[X_P]/(f).
\]
This shows that $pq + r^2$
is a scalar multiple of $f$.  Since $D\neq D'$, $p$ and $q$ are distinct.
\end{proof}
  
The proof of Theorem~\ref{Thm:complexcount} requires a characterization of when $\gamma=0$ in Lemma~\ref{lem:Divisor2Biform}. 
In fact, the next series of lemmata show that for \emph{generic} $f$ and any two-torsion point $[E]$, the resulting constant $\gamma$ is non-zero. 

The curve $C$ is hyperelliptic. If $X_P$ is not the Segre surface, we can identify
  the double cover $\psi\colon C \to \P^1$ in this toric embedding by
  the unique ruling of the scroll. Two points $p_1$ and $p_2$ of the curve
  $C$ satisfy $\psi(p_1)=\psi(p_2)$ if and only if the line
  $\ol{p_1p_2}$ is contained in the ruling of $X_P$. All but finitely many lines of the ruling
  intersect the quadric $\V(f)$ in two distinct points. The
  ramification points of $\psi$ are exactly the intersection points
  with lines on $X_P$ that are tangent to the quadric $\V(f)$. 
  
\begin{Lem}\label{Lem:zeroSOS}
  Let $p,q,$ and $r \in \C[X_P]_1$ such that $p\neq q$ and
  $pq + r^2 = 0$ in $\C[X_P]$. The curves $\V(p)\cap X_P$ and
  $\V(q)\cap X_P$ have a common irreducible component. Furthermore,
  the curve $\V(p)\cap X_P$ contains a line of the ruling of $X_P$
  with multiplicity at least two.
\end{Lem}

\begin{proof}
  This is most easily expressed in terms of biforms.
  Suppose $p,q,r \in \C[x,y,s,t]$ are biforms with Newton polytope $P$ satisfying $p\neq q$ 
  and $pq + r^2 = 0$. Write $p = p_1 p_2$, $q = q_1q_2$,
  $r = r_1 r_2$, where $p_1, q_1, r_1\in \C[s,t]$ are bivariate forms
  in $s$ and $t$ and $p_2, q_2, r_2$ have no factors in
  $\C[s,t]$. Since $p_2, q_2, r_2$ have degree one in the variables
  $x,y$ and no factors of degree $0$ in $x,y$, they must be
  irreducible in $\C[s,t,x,y]$. It follows that each of $p_2$, $q_2$, and $r_2$ 
  are relatively prime to each of $p_1$, $q_1$, and $r_1$. 
  The factorization $( p_1 p_2) ( q_1 q_2) = - (r_1 r_2)^2$ then 
  implies that  $p_2 = q_2 = r_2$ and 
  $p_1q_1 = -r_1^2$.  Since $p_1, q_1, r_1 \in \C[s,t]$, we can factor
  $ p_1 = u^2 w$, $q_1 = v^2 w$, and $r_1 = i u v w$, for some
  $u, v,w \in \C[s,t]$.  The common factor $wp_2$ of all three biforms
  corresponds to a common irreducible component of the curves
  $\V(p)\cap X_P$, $\V(q)\cap X_P$, and $\V(r)\cap X_P$. 
  The assumption that $p\neq q$ implies that the bivariate form $u$ is not a constant and has some root $[s:t] \in \P^1$. 
  This root corresponds to a line in $\V(p)\cap X_P$. Since $u^2$ divides $p$, it has multiplicity at least two.
\end{proof}

We show that for \emph{generic} $f$, the linear form $p$ in Lemma~\ref{Lem:zeroSOS} cannot 
define an even divisor on the curve $C$ defined by $f$.  This non-generic condition is equivalent to 
another condition on the representations $f = p_0q_0 + r_0^2$. 

\begin{Lem}\label{Lem:ZeroIff}
Suppose that the quadratic form $f\in \C[X_P]_2$ defines a smooth curve $C\subset X_P$. Then the following statements are equivalent.
\begin{itemize}
\item[(a)] There exist linear forms $p,q,$ and $r \in \C[X_P]_1$ such that $p\neq q$, $pq+r^2=0$ in $\C[X_P]_2$,
and the divisor $\div_C(p)$ is even.
\item[(b)] For every representation $f = p_0q_0+r_0^2$, the linear system
  $|\frac12 \div_C(p_0)|$ contains a divisor of the form $R_1+R_2+G$,
  where $R_1$ and $R_2$ are ramification points of the double cover
  $\psi\colon C\to \P^1$ and $G$ is an effective divisor.
\end{itemize}
Moreover, a generic quadratic form $f$ in $\C[X_P]_2$ does not satisfy
these conditions.
\end{Lem}

\begin{proof}
  We first show the equivalence of (a) and (b). Suppose that
  $pq+r^2=0$ in $\C[X_P]$ where $\div_C(p)$ is even and
  $f = p_0q_0 + r_0^2$. Note that the divisor
  $\frac12 \div_C(p) - \frac12 \div_C(p_0)$ is a two-torsion point. It
  is linearly equivalent to $[R_2-R_1]$, where $R_1$ and $R_2$ are
  ramification points of $\psi$, see \cite[Section
  5.2.2]{DolMR2964027}. By Lemma \ref{Lem:zeroSOS}, the curve
  $\V(p)\cap X_P$ contains a line of the ruling with multiplicity two,
  so $\frac12 \div_C(p) = P_1 + P_2 + G$ with $\psi(P_1) = \psi(P_2)$.
  The divisor $P_1+P_2$ on $C$ is linearly equivalent to $2R_2$
  because $\psi(P_j)$ is linearly equivalent to $\psi(R_2)$ on $\P^1$.
  Thus
\[
\frac12 \div_C(p_0) \ \sim  \ \frac12 \div_C(p) + R_1 - R_2 \ \sim  \ R_1 + R_2 +G.
\]
Conversely, suppose the linear system $|\frac12 \div_C(p_0)|$ contains a
divisor of the form $R_1+R_2+G$ for the representation $f = p_0q_0+r_0^2$.
Then the linear system $|\frac12 \div_C(p_0) + R_1-R_2|$ contains the
divisor $2R_1 + G$ and therefore also $2R_2 + G$. 
As in the proof of Lemma~\ref{lem:Divisor2Biform},
since $C$ is projectively normal and $[R_1-R_2]$ is a two-torsion point, there are
linear forms $p$, $q$, and $r$ such that $\div_C(p) = 4R_1 + 2G$,
$\div_C(q) = 4R_2 + 2G$, and $\div_C(r) = 2R_1 + 2R_2 + 2G$.
Therefore, the rational function $\frac{pq}{r^2}$ is constant on
$C$. After rescaling, we obtain an identity $pq+r^2 = \gamma f$
in $\C[X_P]$, for some $\gamma\in\C$. If $\gamma$ were non-zero, every
point in the support of $G$ would be a singular point of $C$,
a contradiction.

Finally, we prove the genericity statement. Note first that if $X_P$
is the Segre surface ($d=e=1$), its defining ideal does not contain
any quadratic form
of rank three, hence (1) cannot occur. So we assume that $X_P$
is not the Segre surface. Geometrically, statement (2)
means that $\V(p_0)\cap X_P$ contains two lines of the ruling. To see
this, note that the hyperplane section of $C$ defined by $p_0$ is even,
so $\V(p_0)$ is tangent to $C$ at $R_1$. Since $R_1$ is a ramification
point of the double cover $\psi\colon C\to \P^1$, the tangent to $C$
at $R_1$ is a line of the ruling of $X_P$, which is therefore
contained in $\V(p_0)$.

Lines in $X_P$ are skew and form a one-dimensional family, so the
  variety $\sH$ of linear forms $p\in \C[X_P]_1$ for which $\V(p)$ contains two
  lines in $X_P$ has codimension two in $\C[X_P]_1$.  Consider the map
\[
\phi' \colon \left\{
\begin{array}[h]{ccc}
\sH \times \C[X_P]_1 \times \C[X_P]_1 & \to & \C[X_P]_2 \\
(p,q,r) & \mapsto & pq + r^2.
\end{array}\right.
\]
Its differential at $(p,q,r)$ maps $h = (h_1,h_2,h_3)$ to the
quadratic form $ph_2 + qh_1 + 2rh_3$, where $h_1$ is taken from the
tangent space to $\sH$ at $p$. This tangent space has codimension two.
Hence the rank of ${\rm d}\phi'$ is at most $3 \dim(\C[X_P]_1) - 4$, because
$h = (0,r,-\frac12 p)$ and $h = (p,-q,0)$ lie in the kernel.  Note
that $p$ is in the tangent space to $\sH$ at $p$ because $\sH$ is a
cone.  The space $\C[X_P]_2$ has dimension $3\dim(\C[X_P]_1)-3$. This comes from 
\eqref{eq:dimcount} on page \pageref{eq:dimcount} with $m=2$, $n = d+e+1$ 
and the observation that  $\dim_\R(\R[X_P]_k) = \dim_\C(\R[X_P]_k\otimes \C)$ for any $k\in \Z_{\geq 0}$.
Since ${\rm d}\phi'$ maps generically onto the tangent space of the image of
$\phi'$, that image must be contained in a hypersurface.
\end{proof}

\begin{Rem}\label{rem:ruling}
The condition in Lemma~\ref{Lem:ZeroIff} that $|\frac12
\div_C(p)|$ contains a divisor of the form $R_1+R_2+G$ for a
representation $f = pq+r^2$ can be expressed in terms of biforms. 
If $\frac12 \div_C(p)$ actually equals $R_1+R_2+G$, then the biform corresponding to $p$ is divisible 
by the two linear forms in $\C[s,t]$ defining $\psi(R_1)$ and $\psi(R_2)$ in $\P^1$. 
\end{Rem}

By excluding the non-generic quadratic forms $f$ described in Lemma~\ref{Lem:ZeroIff}, we can count sums-of-squares representations and prove Theorem~\ref{Thm:complexcount}. 

\begin{proof}[Proof of Theorem~\ref{Thm:complexcount}]\label{proof:complexcount}
We establish a bijection between inequivalent representations of $f$ as a sum of three squares and two-torsion points in the 
Jacobian of the curve $C$, \textit{i.e.} divisor classes $[E]$ with $2E\sim 0$. 
The Jacobian is a $g$-dimensional complex torus, 
therefore the number of two-torsion points is $2^{2g}$ (see e.g.~\cite[Section 5.2.2]{DolMR2964027}).

As in the discussion above Lemma~\ref{lem:Divisor}, over $\C$, representations of $f$ as a sum of three squares 
correspond to representations $f = pq + r^2$. 
We fix a representation $f = p_0 q_0 + r_0^2$ of $f$, which exists by
Theorem~\ref{Thm:varmin}.  For every two-torsion point $[E]$, Lemma~\ref{lem:Divisor2Biform} gives 
a representation $\gamma \cdot f = pq+r^2$ where $\gamma \in \C$  for which 
$\div_C(p)$ is even, and $[\frac{1}{2}\div_C(p)-  \frac{1}{2}\div_C(p_0)]$ equals $[E]$. By Lemma~\ref{Lem:ZeroIff} and 
the genericity of $f$, $\gamma$ is nonzero and we can rescale $p,q,r$  so that $f = pq + r^2$. 
Thus for generic $f$, every two-torsion point
gives rise to a representation $f = pq+r^2$. 

Conversely, from a representation
$f = pq+r^2$ we obtain the two-torsion point $[E]$, where
\[
E  \ = \  \frac12 \div_C(p) - \frac12 \div_C(p_0).
\]
The two maps just constructed are inverses of each other, up to equivalence of representations and divisors, respectively. 
By Proposition~\ref{prop:equivalence} below, these notions of equivalence are compatible, giving a bijection between equivalence classes of representations as a sum of three squares and two-torsion points of the Jacobian of $C$. 
\end{proof}

We now discuss the equivalence of representations as sums of squares in relation to equivalence of divisors on $C$.

\begin{Prop}\label{prop:equivalence}
Two representations $f = pq+r^2$ and $f = p'q'+(r')^2$ are equivalent, meaning $\lspan_\C(p,q,r) = \lspan_\C(p',q',r')$ in the space of linear forms, if and only if 
\[
\frac12 \div_C(p) \ \sim  \ \frac12 \div_C(p').
\]
\end{Prop}

\begin{proof}
The $2\times 2$ matrices
$Q = \begin{pmatrix} p & ir \\ ir & q \end{pmatrix}$ and $Q' = \begin{pmatrix} p' & ir '\\ ir' & q' \end{pmatrix}$
are determinantal representations of $f$, \textit{i.e.} $f = \det(Q) = \det(Q') $. 
The matrix $Q$ defines a surjective morphism $\phi_Q:\P^1\times\P^1 \rightarrow \P(\lspan_\C(p,q,r))$ 
given by $(\lambda, \mu) \mapsto \lambda^T Q \mu$, and the image of the diagonal 
$\Delta \subset \P^1\times \P^1$ is a conic in this plane, of which every point 
$\lambda^T Q\lambda$ defines an even divisor on $C$. To see this, note that for any $\mu \neq \lambda \in \P^1$, the
determinant of the matrix $(\lambda \ \mu)^T Q (\lambda \ \mu)$ is a scalar multiple of $\det(Q) = f$. 
Thus 
\[\div_C(\lambda^T Q \lambda) + \div_C(\mu^T Q \mu)  \ = \  2\cdot \div_C(\lambda^T Q \mu). \]
Furthermore, the divisors $\frac{1}{2}\div_C(\lambda^T Q \lambda)$ and
$\frac{1}{2}\div_C(\mu^T Q \mu)$ are equivalent. For example, taking
$\lambda =[1:0]$ and $\mu =[0:1]$, we see that
\[
\frac12 \div_C(p) \  \sim  \  \frac12 \div_C(p) + \div_C(r/p)  \ = \  \div_C(r) - \frac12 \div_C(p)  \ = \  \frac12 \div_C(q).
\]

Now suppose that $\lspan_\C(p,q,r) = \lspan_\C(p',q',r')$.
The maps $\phi_Q$ and $\phi_{Q'}$ then map $\P^1\times \P^1$ to the same plane. The images $\phi_Q(\Delta) , \phi_{Q'}(\Delta)$ 
of the diagonal are conics in this plane and therefore must intersect. So for some $\lambda, \mu \in \P^1$, 
$\lambda^T Q\lambda =\mu^T Q'\mu$. Then
\[ \frac12 \div_C(p) \ \sim \ \frac12 \div_C(\lambda^T Q\lambda) \ = \ \frac12 \div_C(\mu^T Q'\mu) \ \sim \ \frac12 \div_C(p').\] 

Conversely, suppose $f=pq+r^2=p'q'+(r')^2$ where
$\frac12 \div_C(p) \sim \frac12 \div_C(p')$. Then there exists a
rational function $h\in\C(C)^\ast$ such that
$\frac 12\div_C(p)=\frac 12\div_C(p')+\div(h)$, where $\div(h)$ now denotes the principal divisor defined by the zeros and poles of $h$. It follows that
$\frac 12\div_C(p)+\frac 12\div(p')=\div_C(p')+\div(h)$. Since
$C\subset\P^{d+e+1}$ is projectively normal, this implies the
existence of a linear form $s$ such that
$\div_C(p)+\div_C(p')=\div_C(s^2)$. After rescaling, 
this gives a representation $f=pp'+s^2$. 

For $\lambda, \mu\in \P^1$, we have 
$\frac 12\div_C(\lambda^T Q \lambda ) \sim \frac12 \div_C(p) \sim \frac12 \div_C(p')\sim \frac 12\div_C(\mu^T Q' \mu )$, 
so by the same argument there is a representation $f=(\lambda^T Q \lambda )(\mu^T Q' \mu ) + (s_{\lambda \mu})^2$.
This defines a map from $\P^1\times \P^1$ into the Grassmannian of planes in $\P^{d+e+1}$ 
that maps $(\lambda, \mu)$ to $\lspan_\C\{\lambda^T Q \lambda, \mu^T Q' \mu, s_{\lambda \mu}\}$.
The image is either infinite or a single point. 
Each point in the image corresponds to the column span of a rank-three Gram matrix of $f$. By Corollary~\ref{cor:finitereps}, 
there are only finitely many such Gram matrices of $f$.

It follows that $L = \lspan_\C\{\lambda^T Q \lambda, \mu^T Q' \mu, s_{\lambda \mu}\}$ does not depend on $\lambda, \mu$. 
Since $\lspan_\C\{\lambda^T Q \lambda : \lambda \in \P^1\}$ equals $\lspan_\C\{p,q,r\}$ and is contained in $L$, 
we must have that $L = \lspan_\C\{p,q,r\}$. By similar arguments, $L$ contains $\lspan_\C\{\mu^T Q' \mu : \mu \in \P^1\}$
and must equal $\lspan_\C\{p',q',r'\}$.
\end{proof}

\begin{Rem}
The group of matrices $U$ that leave the quadratic form
$(p,q,r)\mapsto pq+r^2$ invariant is conjugate to the group of orthogonal $3\times 3$ matrices that define the equivalence relation on sums of three squares $\ell_1^2 + \ell_2^2 + \ell_3^2$.
Namely, each $U$ can by written as $AO A^{-1}$, where $O O^t = I_3$ and 
\[
A  \ = \  
\begin{pmatrix}
1 & i & 0 \\
1 & -i & 0 \\
0 & 0 & 1 \\
\end{pmatrix}. 
\]
Two representations $f = pq+r^2$ and $f = p'q'+(r')^2$
are equivalent, as in Proposition~\ref{prop:equivalence}, if and only if 
they are related by a linear relation $(p,q,r)^T = U (p',q',r')^T$.
\end{Rem}

\begin{Exm}\label{Exm:generic}
The genericity assumptions in Theorem~\ref{Thm:complexcount} are necessary. Indeed, 
consider the following biform of degree $(2,4)$, which defines a smooth curve in $\P^1\times \P^1$:
\[ f \ = \  (x-y) (s^2 - t^2)\cdot (x+y)(s^2 - 9 t^2) + x^2 (s^2 - 4 t^2)^2.\]
This example violates the condition of genericity of Lemma~\ref{Lem:ZeroIff}. For example, 
we have an expression $p_0q_0 + r_0^2$, where  $p_0 = (x-y)(s^2-t^2)$ is divisible by two linear forms
$(s+t)$ and $(s-t)$.  These define ramification points $R_1, R_2$ on the curve $C$. 
The curves $\V(f)$ and $\V(p_0)$ are shown in the affine chart $\{t=1, y=1\}$ in Figure~\ref{fig:nonGen}.

Adding the two-torsion point $[R_1-R_2]$ to $[\frac{1}{2}\div_C(p_0)]$ gives a divisor class $[\frac{1}{2}\div_C(p)]$, 
defined by the biform $p = (x-y)(s+t)^2$.    
This fails to produce a representation $\gamma f = pq - (r)^2$ with $\gamma\neq 0$.  
Instead, by taking $q = (x-y)(s-t)^2$ and $r = (x-y)(s+t)(s-t)$, we find such a representation with $\gamma = 0$. 
In fact, $f$ has only $60$ representations as a sum of three squares, which is four 
fewer than a generic biform of degree (2,4). 
\end{Exm}

\begin{center}
 \begin{figure}[h]
\begin{center}
 \includegraphics[height = 2in]{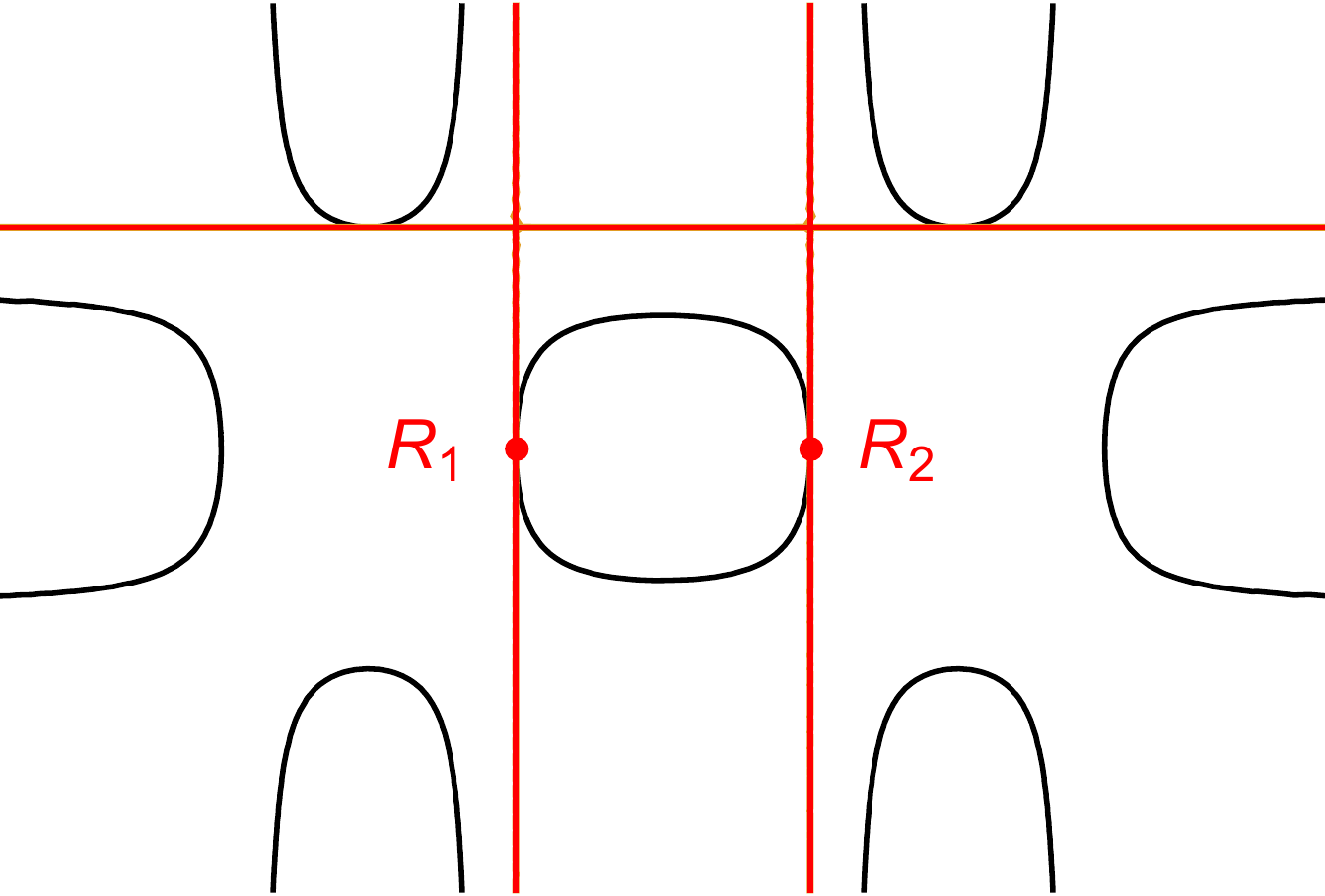}
 \end{center}
\caption{The non-generic curves $\V(f)$ and $\V(p_0)$ from Example~\ref{Exm:generic}.  }
\label{fig:nonGen}
\end{figure}
\end{center}

We now would like to identify the representations of $f$ as a sum and difference of squares of real polynomials. These will correspond to
the real two-torsion points on the Jacobian, which are divisor classes
$[D]$ such that $D\sim\ol{D}$ and $2D\sim 0$. 

For the discussion of real sums and differences of squares, we find it useful
to use the language of Gram matrices. 
As discussed in Definition~\ref{def:Equiv}, rank-three Gram matrices of a polynomial $f$ correspond to 
equivalence classes of representations of $f$ as a sum of three
squares.  Equivalence classes of representations as a sum of squares over $\R$ 
correspond to positive semidefinite Gram matrices, whereas equivalence classes of representations
as a sum and difference of squares over $\R$ correspond to indefinite Gram matrices.

\begin{Thm}\label{Thm:realsosrep}
  Let $f$ be a generic positive polynomial with Newton polytope
  $2P_{d,e}$ and let $g=d+e-1$. 
  The $f$ has $2^g$ positive semidefinite Gram matrices of rank three, which 
  correspond to $2^g$ inequivalent representations of $f$ as a sum of three real squares. 
  If $g$ is even, then there are no other real Gram matrices of rank three. 
  If $g$ is odd, $f$ has an additional $2^g$ real indefinite Gram
  matrices of rank three, which 
  correspond to representations of $f$ as a sum and difference of three real squares. 
\end{Thm}

\begin{proof}
  By Theorem~\ref{Thm:varmin}, there exists a representation of $f$ as
  a sum of three squares, say $f = p_0\ol{p}_0 + r_0^2$, where $p_0$
  and $r_0$ are polynomials whose Newton polytope is contained in
  $P_{d,e}$, $p_0$ has complex coefficients and $r_0$ has real
  coefficients. As in the proof of Theorem~\ref{Thm:complexcount},
  this fixed representation gives a bijection between all representations
  over the complex numbers and two-torsion points of the Jacobian of
  $C$, the smooth model of $\V(f)$ constructed in Lemma~\ref{lem:curve}.

Let $D_0=\frac 12\div(p_0)$ be the divisor associated with the given
representation. If $D_1=\frac 12\div(p_1)$ comes from any other
representation of $f$ as a sum of three squares, then, since
$D_0\sim\ol{D_0}$ and $D_1\sim \ol{D_1}$, the two-torsion point 
$[D_0-D_1]$ is an element of $J(\R)$, the real points of the Jacobian
$J$ of $C$. 

For the converse direction, we introduce the
following notation: Given any divisor $D$ on $C$ with $D\sim\ol D$,
pick a rational function $h_D\in\C(C)^{\ast}$ with
$\div(h_D)=\ol{D}-D$. Then $\div(h_D\ol{h_D})=0$, so that
$h_D\ol{h_D}$ is a non-zero real constant, which we denote by
$c_D$. The constant $c_D$ depends on the choice of $h_D$, but a simple
computation shows that its sign only depends on the linear equivalence
class of $D$. Indeed, if $D'\sim D$, say $D'=D+\div(g)$, $\div(h_D)=\ol{D}-D$
and $\div(h_{D'})=\ol{D'}-D'$, then
$\div(h_{D'})=\ol{D}-D+\div(\ol{g}/g)=\div(h_D\cdot \ol{g}/g)$. Thus
there is $\gamma\in\C^*$ with $h_{D'}=\gamma h_{D}\ol{g}/g$, so that
$h_{D'}\ol{h_{D'}}=|\gamma|^2 h_{D}\ol{h_D}$.
This defines a group homomorphism
\[
({\rm Cl}\ C)(\R)\to\{\pm 1\},\ \ \ [D]\mapsto {\rm sgn}(c_D)
\]
from the group of conjugation-invariant divisor classes on $C$ into
the multiplicative group $\{\pm 1\}$.

Now let $[E]\in J(\R)$ be a real two-torsion point represented by a divisor
$E\sim\ol{E}$. As in the proof of Theorem~\ref{Thm:varmin} and Lemma~\ref{lem:Divisor2Biform}, we find
$l(D_0+E)\ge 2$, and we pick an effective divisor $D\in
|D_0+E|$ with $D\neq \ol{D}$; we further pick a complex linear
form $p$ and a real linear form $r$ satisfying
\[
\div_C(p)=2D,\ \ \ \div_C(r)=D+\ol{D}.
\]
Choosing $h_D$ as above, we find $\div(h_D)=\div(r)-\div(p)$, hence $
r/p$ equals $\alpha h_D$
for some $\alpha\in\C^\ast$. It follows that
\[
\frac{r^2}{p\ol{p}} \ = \ \frac{r}{p}\frac{\ol{r}}{\ol{p}} \ = \ (\alpha h_D)(\overline{\alpha h_D}) \ = \ |\alpha|^2c_D.
\]
We conclude that after rescaling 
\[
f \ = \ r^2-c_D|\alpha|^2p\ol{p}.
\]
This shows that signed representations of $f$ as a sum of three
squares are in bijection with real two-torsion points.
It remains to determine for which choices of $[E]\in J_2(\R)$ the
constant $c_D$ has negative sign. We have $E\sim D-D_0$, hence
\[
{\rm sgn}(c_D) \ = \ {\rm sgn}(c_{D_0})\cdot{\rm sgn}(c_{E}).
\]
Since $D_0$ comes from a real sum-of-squares representation, we know
that ${\rm sgn}(c_{D_0})=-1$. Hence we get another real sum-of-squares
representation from $D$ if and only if $[E]$ lies in the kernel of the sign map, \emph{i.e.} ${\rm sgn}(c_E)=1$. The count
we claim now follows from \cite[Prop.~6.5 and Lemma 6.8]{ScheidererMR2580677}. In fact, it is shown there more precisely that the kernel
of the above sign map, restricted to $J(\R)$, is given exactly by those
conjugation-invariant divisor classes that contain a
conjugation-invariant divisor.
\end{proof}

Here are two examples illustrating the counts in Theorems~\ref{Thm:complexcount}~and~\ref{Thm:realsosrep}.
\begin{Exm} ($e=d=1, g=1$) 
The polytope $P = P_{1,1}$ defines the toric variety $X_P= \P^1\times \P^1$, with the usual embedding into $\P^3$. 
Consider the curve $\V(f)\subset \P^1\times \P^1$ defined by $f = x^2(t^2 + s^2) +y^2 (2t^2+ 2 s t + 2 s^2)$. 
This is birational to the smooth projective curve $C$ obtained as the intersection of $X_P$ with a quadric in $\P^3$. 
The space of Gram matrices of $f$ is one-dimensional, as for any
$\alpha\in \C$, we have 
\[ 
f(s,t,x,y) \ \ = \ \
\begin{pmatrix}
yt \\ ys \\ xt\\ xs
\end{pmatrix}^T \begin{pmatrix}
2&1&0&\alpha\\
1&2&-\alpha&0\\
0&-\alpha&1&0\\
\alpha&0&0&1\\
\end{pmatrix}
\begin{pmatrix}
yt \\ ys \\ xt\\ xs
\end{pmatrix}.
\]
This Gram matrix has rank three for $\alpha \in \{\pm 1, \pm \sqrt{3}\}$, giving $4 = 2^{2g}$
inequivalent representations of $f$ as a sum of three squares over $\C$. Since $g$ is odd, 
$4 = 2^{g+1}$ rank-three Gram matrices are real, $2 = 2^g$ are positive semidefinite (for $\alpha=\pm1$), and
$2 = 2^g$ are indefinite (for $\alpha = \pm\sqrt{3})$. For example,
$\alpha= 1$ and $\alpha=\sqrt{3}$ give
\[f \ = \ (x t - s y)^2 + (x s + y t)^2 + (y t + ys )^2 \ = \  (x t - \sqrt{3}s y)^2 + (x s + \sqrt{3} y t)^2 - (y t - ys )^2. \]
Both $p_0 = (x t - s y) + i(x s + y t)$ and $p_1= (x t - \sqrt{3}s y)+i(x s + \sqrt{3} y t)$ define even divisors on the curve $C$, 
say $2D_0$ and $2D_1$. Then $[E] = [D_1 - D_0]$ is a real two-torsion point of the Jacobian, as $E \sim \ol{E}$ and $2E \sim 0$. 
Since $(y t + ys )^2/p_0 \ol{p_0} $ and $ (y t - ys )^2/p_1 \ol{p_1}$ 
are both constant on $\V(f)$ with opposite signs, $[D_0]$ and $[D_1]$ have different images under the group homomorphism $({\rm Cl} \ C)(\R) \rightarrow \{\pm1\}$ used in the proof of Theorem~\ref{Thm:realsosrep}.
\end{Exm}

\begin{Exm} ($e=1, d=g=2$)
Consider $f= t^2 (t^2 + s^2) x^2+(t^4 + t^2 s^2 + s^4)y^2 $.
This polynomial defines a singular curve in $\P^1\times \P^1$ which 
has a smooth model $C$ in the toric variety $X_{P} \into \P^4$ defined by the polytope $P = P_{2,1}$.
For all $(\alpha, \beta, \gamma)\in \C^3$, 
\[
f(s,1, x,1)   \ \ =  \ \
\mbox{\small
$\begin{pmatrix}
1 \\  s  \\ s^2 \\ x\\ xs
\end{pmatrix}^T
\begin{pmatrix}
1 & 0 & \alpha & 0 & \beta \\
 0 & 1-2  \alpha  & 0 & -\beta & \gamma \\
  \alpha  & 0 & 1 & -\gamma & 0 \\
 0 & -\beta & -\gamma & 1 & 0 \\
 \beta & \gamma & 0 & 0 & 1 \\
 \end{pmatrix}
 \begin{pmatrix}
1 \\  s  \\ s^2 \\ x\\ xs
\end{pmatrix}$.}
\]
There are a total of $16 = 2^{2g}$ points $(\alpha, \beta, \gamma)\in \C^3$ for which 
this matrix has rank three. Of these sixteen rank-three Gram matrices, only $4 = 2^g$ are real and each of these four are positive semidefinite.
For example, $(\alpha, \beta, \gamma) = (0,0,1)$ corresponds to the sum-of-squares representation
$f(s,1, x,1)= (1)^2 + (s+xs)^2 + (s^2 - x )^2$.
\end{Exm}

The above counts are for smooth surfaces of minimal degree. The
classification of varieties of minimal degree tells us that every
singular variety of minimal degree is a cone over a smooth one. We can
count the representations on singular varieties of minimal degree by
completing the square.
\begin{Lem}\label{lem:cone}
Let $X\subset \P^{n-1}$ be an $(m-1)$-dimensional variety of minimal degree, and let $Y\subset \P^{n}$ be a 
cone over $X$.  
The number of Gram matrices of rank $m+1$ of a generic positive quadratic form on $Y$ 
equals the number of Gram matrices of rank $m$ of a generic positive quadratic form on $X$. 
Moreover, this equality holds under restriction to real Gram matrices and to real positive semidefinite Gram matrices. 
\end{Lem}

\begin{proof}
We can choose coordinates $[y:x_1 : \hdots: x_n]$ on $\P^n$ so that $X\subset \V(y)$ and 
$Y$ is the cone over $X$ from the point $ (1:0:\hdots:0)$. 
A generic positive quadratic form $f$ on $Y$ can be written $f = a y^2 + 2 b y + c$ where $a\in \R$, $b\in \R[X]_{1}$ and $c\in \R[X]_{2}$. 
By the genericity and positivity of $f$, we can take $a>0$.  
Then any Gram matrix $G$ of $f$ can be written as 
\[G\  = \ \begin{pmatrix}
a & B^T \\
B & C 
\end{pmatrix} 
\ =\
 \begin{pmatrix}
a & B^T \\
B & a^{-1}BB^T 
\end{pmatrix} 
+ 
  \begin{pmatrix}
0 & 0 \\
0 & C - a^{-1}BB^T 
\end{pmatrix} 
\]
where $B$ is the vector of coefficients of $b$ and $C$ is some Gram matrix of $c$. 
Then $G' = C - a^{-1}BB^T$ is a Gram matrix 
of $c - b^2/a \in \R[X]_{2}$. By the rank-additivity properties of Schur complements, the
rank of $G$ is $\rank(G') +1$. 
As $a$ and $B$ are fixed, the map $G\mapsto G'$ provides a bijection between 
rank-$r$ Gram matrices of $f$ and rank-$(r-1)$ Gram matrices of $c - b^2/a$
that preserves reality and positive semi-definiteness. 
Furthermore the genericity of $f\in \R[Y]_2$ ensures the genericity of $c-b^2/a \in \R[X]_2$. 
\end{proof}

\begin{Thm} \label{thm:allSurfaces}
Let $X \subset \P^n$ be a nondegenerate irreducible surface of minimal degree with dense real points. Then a generic quadratic form nonnegative on $X$ has exactly $2^{n-2}$ inequivalent representations as a sum of three squares.
\end{Thm}
\begin{proof}
By the characterization of varieties of minimal degree, a surface of minimal degree is either a quadratic hypersurface in $\P^3$, 
the quadratic Veronese embedding of $\P^2$ in $\P^5$, a smooth rational normal scroll, or 
the cone over a rational normal curve. Positive quadratic forms on a quadratic hypersurface in $\P^3$ with dense real points have two inequivalent representations as sums of three squares.
Positive quadratic forms on the quadratic Veronese of $\P^2$ correspond to positive ternary quartics, which
have $8 =2^{5-2} $ inequivalent representations as a sum of three squares \cite{PowersReznickScheidererSottileMR2103198}. 
Theorem~\ref{Thm:realsosrep} addresses the case of quadratic forms on a smooth rational normal scroll. 

Finally, suppose $X\subset \P^n$ is the cone over the rational normal
curve $\nu_{d}(\P^1)$ with $d = n-1$.  By Lemma~\ref{lem:cone},
representations of a quadratic form on $X$ as a sum of three squares
correspond to representations of a quadratic form on $\nu_{d}(\P^1)$
as a sum of two squares.  As quadratic forms on $\nu_{d}(\P^1)$ are
bivariate forms of degree $2d$, we see from \cite[Example
2.13]{ChoiLamReznickMR1327293} that there are $2^{n-2}$ 
inequivalent such representations.
\end{proof}

\begin{Exm}\label{exm:repsforsurf} ($\dim(X)=2$, ${\rm codim}(X) = 3$). 
It is remarkable that the number of sums-of-squares representations is more regular over $\R$ than over $\C$. 
We illustrate this in the case when $X$ is a surface of minimal degree in $\P^5$. 
By the classification of varieties of minimal degree, an irreducible nondegenerate surface of minimal degree in $\P^5$ is projectively equivalent to 
one of four toric varieties: a cone over the quartic rational normal curve, the second Veronese of $\P^2$, or the rational normal 
scrolls $X_{P_{2,2}}$, $X_{P_{3,1}}$. The number of rank-three Gram matrices of each type are shown in Table~\ref{table:genus3}.

\begin{center}
\begin{table}[h]
\begin{center}
\begin{tabular}{c|c|c|c}
\ Surface \ \ 	&	\ \ real $\&$ psd \ \	& \ \	real \ \	&  \ \ complex \  \\
\hline
${\rm cone}(\nu_{4}(\P^1))$	&	8		&	11			& 	35	\\
$\nu_2(\P^2)$			&	8		&	15			& 	63	\\
$X_{P_{2,2}}$			&	8		&	16			& 	64	\\	
$X_{P_{3,1}}$	&			8		&	16			& \	64	\smallskip\\
\end{tabular}
\end{center}
\caption{The number of rank-three Gram matrices of a generic positive quadratic form on surfaces of minimal degree in $\P^5$. }\label{table:genus3}
\end{table}
\end{center}
The counts for $X_{P_{2,2}}$ and $X_{P_{3,1}}$ follow from Theorem~\ref{Thm:realsosrep}. 
Positive quadratic forms on $\nu_2(\P^2)$ correspond to ternary quartics, studied in \cite{PowersReznickScheidererSottileMR2103198}.
Finally, by Lemma~\ref{lem:cone}, rank-three Gram matrices of a generic quadratic form on the cone over $\nu_4(\P^1)$ correspond to rank-two Gram matrices of a generic quadratic form on $\nu_4(\P^1)$.
\end{Exm}

\end{document}